\documentclass[11pt,reqno]{amsart}
\usepackage{amsmath}
\usepackage{amsthm}
\usepackage{amsfonts}
\usepackage{amssymb}
\usepackage{bm}
\usepackage{tikz}
\usepackage{float}
\usepackage[all]{xy}
\usepackage{hyperref}
\usepackage{enumerate}
\usepackage{enumitem}
\usepackage[english]{babel}
\usepackage{mathtools}

\newcommand{\op}{\ensuremath{^{\mathrm{op}}}}

\newcommand{\bZ}{\mathbb{Z}}

\newcommand{\cA}{{\mathcal A}}
\newcommand{\cB}{{\mathcal B}}
\newcommand{\cC}{{\mathcal C}}

\newcommand{\cG}{{\mathcal G}}

\newcommand{\cP}{{\mathcal P}}
\newcommand{\cR}{{\mathcal R}}
\newcommand{\cS}{{\mathcal S}}

\newcommand{\cV}{{\mathcal V}}

\newcommand{\cX}{{\mathcal X}}

\newcommand{\md}{\operatorname{mod}}
\newcommand{\Md}{\operatorname{Mod}}
\newcommand{\add}{\operatorname{add}}
\newcommand{\im}{\operatorname{im}}

\newcommand{\Proj}{\operatorname{Proj}}
\newcommand{\Gproj}{\cG\cP}

\newcommand{\emphbf}[1]{\emph{\textbf{#1}}}
\newcommand{\pdim}{\operatorname{proj.dim}}
\newcommand{\idim}{\operatorname{inj.dim}}

\newcommand{\domdim}{\operatorname{dom.dim}}

\newcommand{\Hom}{\operatorname{Hom}}

\newcommand{\Ext}{\operatorname{Ext}}

\newcommand{\Ker}{\operatorname{Ker}}
\newcommand{\Coker}{\operatorname{Coker}}

\usetikzlibrary{matrix,arrows}
\title{Co-Gorenstein algebras}

\date{\today}

\author{Sondre Kvamme}
\address{(Kvamme) Laboratoire de Math\'ematiques d'Orsay, Univ. Paris-Sud, CNRS, Universit\'e Paris-Saclay, 91405 Orsay, France}
\email{sondre.kvamme@u-psud.fr}

\author{Ren\'{e} Marczinzik}
\address{(Marczinzik) Institute of algebra and number theory, University of Stuttgart, Pfaffenwaldring 57, 70569 Stuttgart, Germany}
\email{marczire@mathematik.uni-stuttgart.de}

\begin{document}

\newtheorem{Theorem}[equation]{Theorem}
\newtheorem{Lemma}[equation]{Lemma}
\newtheorem{Corollary}[equation]{Corollary}
\newtheorem{Proposition}[equation]{Proposition}
\newtheorem{Conjecture}[equation]{Conjecture}

\theoremstyle{definition}
\newtheorem{Definition}[equation]{Definition}
\newtheorem{Example}[equation]{Example}
\newtheorem{Remark}[equation]{Remark}
\newtheorem{Setting}[equation]{Setting}

\thanks{The authors thank Henning Krause and Steffen Koenig for helpful comments. We would also like to thank the referee for helpful comments and suggestions. The first author was supported by a public grant as part of the FMJH}


\begin{abstract}
We review the theory of Co-Gorenstein algebras, which was introduced in \cite{Bel00}. We show a connection between Co-Gorenstein algebras and the Nakayama and Generalized Nakayama conjecture.
\end{abstract}

\maketitle

\setcounter{tocdepth}{2}

Fix a commutative artinian ring $R$ and an artin $R$-algebra $\Lambda$. Let $\md\text{-}\Lambda$ be the category of finitely generated right $\Lambda$-modules, and let
\[
D(-):= \Hom_R(-,I)\colon (\md\text{-}\Lambda)\op \to \md\text{-}\Lambda\op
\]
denote the equivalence where $I$ is the injective envelope of $S_1\oplus S_2\oplus \cdots \oplus S_n$ and $S_1,S_2,\cdots ,S_n$ is a complete set of representatives of isomorphism classes of simple $R$-modules.   Let 
\[
\cdots \to P_{1}(D\Lambda)\to P_0(D\Lambda)\to D\Lambda\to 0
\] 
be a minimal projective resolution of right $\Lambda$-modules and let  
\[
0\to \Lambda\to I_0(\Lambda)\to I_1(\Lambda)\to \cdots
\] 
be a minimal injective resolution of $\Lambda$ as a right module. Recall that the \emphbf{dominant dimension} $\domdim \Lambda$ of $\Lambda$ is the smallest integer $d$ such that $I_d(\Lambda)$ is not projective. We write  $\domdim \Lambda=\infty$ if no such integer exists. The following conjectures are important in the representation theory of artin algebras.

\begin{enumerate}
\item \emphbf{Generalized Nakayama Conjecture (GNC)}: If $P$ is an indecomposable projective right $\Lambda$-module, then $P$ is a summand of $P_n(D\Lambda)$ for some $n$;
\item \emphbf{Nakayama Conjecture (NC)}: If $\domdim \Lambda =\infty$, then $\Lambda$ is selfinjective.
\end{enumerate}

Since $\domdim \Lambda =\infty$ if and only if $P_i(D\Lambda)$ is injective for all $i\geq 0$ by \cite{Hos89}, it follows that \emphbf{GNC} implies \emphbf{NC}.

In this note we show a relation between these conjectures and the notion of a Co-Gorenstein algebra, which was introduced by Beligiannis in \cite{Bel00}. More precisely, we show that there exist implications
\begin{multline*}
\text{\emphbf{GNC}} \xRightarrow{\text{Proposition } \ref{finite injective dimension}} \text{\emphbf{Conjecture} }\ref{Conjecture} \xRightarrow{\text{Proposition } \ref{lemma implies Nakayama lemma}} \text{\emphbf{NC}}. 
\end{multline*} 
where \emphbf{Conjecture} \ref{Conjecture} is as follows:

\begin{Conjecture}\label{Conjecture}
If $\Omega^n(\md\text{-}\Lambda)$ is extension closed for all $n\geq 1$, then $\Lambda$ is right Co-Gorenstein.\footnote{This is Lemma 6.19 part (3) in \cite{Bel00}. Beligiannis claims that it follows immediately from results in \cite{AR94}. However, this is not clear to the authors, so we state it as a conjecture.}
\end{Conjecture}

We start by reviewing the construction and properties of Co-Gorenstein categories. In particular, we give some equivalent properties for an algebra to be right Co-Gorenstein, see Corollary \ref{Alternative description Co-Gorenstein}. In Section \ref{Co-Gorenstein Artin algebras} we show the implications above.

Throughout the note $R$ denotes a commutative artinian ring, $\Lambda$ an artin $R$-algebra. Also, we fix the notation  
\begin{multline*} \Omega^{n}(\md\text{-}\Lambda):= \{M\in \md\text{-}\Lambda\mid\text{there exists an exact sequence} \\
 0\to M\to P_1\to \cdots \to P_n \text{ with } P_i\in \md\text{-}\Lambda \text{ projective for } 1\leq i\leq n \}
 \end{multline*}

\section{Co-Gorenstein categories}\label{Co-Gorenstein categories}

Let $\cA$ be an abelian category with enough projectives, and let $\cP:=\Proj(\cA)$ denote the full subcategory of projective objects in $\cA$. The projectively stable category $\underline{\cA}:= \cA/\cP$ of $\cA$ consists of the same objects as $\cA$, and with morphisms 
\[
\underline{\cA}(A_1,A_2):= \cA(A_1,A_2)/\sim
\]
where $f \sim g$ if $f-g$ factors through a projective object. For a morphism $f\colon A_1\to A_2$ in $\cA$ we let $\underline{f}\colon \underline{A_1}\to \underline{A_2}$ denote the corresponding morphism in $\underline{\cA}$. For each object $A\in \cA$ choose an exact sequence $0\to \Omega A\to P\to A\to 0$ where $P$ is projective. The association $A\mapsto \Omega A$ induces a functor $\Omega \colon \underline{\cA}\to \underline{\cA}$ \cite[Section 3]{Hel60}. Furthermore, if $0\to K\to Q\to A\to 0$ is any exact sequence with $Q$ projective, then there exists a unique isomorphism 
\begin{align}\label{unique syzygy}
\underline{K}\xrightarrow{\cong} \Omega \underline{A}
\end{align} in $\underline{\cA}$ which is induced from a morphism $K\to \Omega A$ in $\cA$ such that there exists a commutative diagram  
\[
\begin{tikzpicture}[description/.style={fill=white,inner sep=2pt}]
\matrix(m) [matrix of math nodes,row sep=2.5em,column sep=4em,text height=1.5ex, text depth=0.25ex] 
{0 & K & Q & A & 0\\
 0 & \Omega A & P & A & 0 \\};
\path[->]
(m-1-1) edge node[auto] {$$} 	    										(m-1-2)
(m-1-2) edge node[auto] {$$} 	    										(m-1-3)
(m-1-3) edge node[auto] {$$} 	    										(m-1-4)
(m-1-4) edge node[auto] {$$} 	    										(m-1-5)

(m-2-1) edge node[auto] {$$} 	    										(m-2-2)
(m-2-2) edge node[auto] {$$} 	    										(m-2-3)
(m-2-3) edge node[auto] {$$} 	    										(m-2-4)
(m-2-4) edge node[auto] {$$} 	    										(m-2-5)

(m-1-2) edge node[auto] {$$} 	    						    			(m-2-2)
(m-1-3) edge node[auto] {$$} 	    									  	(m-2-3)
(m-1-4) edge node[auto] {$1_A$} 	    								 		(m-2-4);	
\end{tikzpicture}
\]
for some morphism $Q\to P$.

\begin{Definition}\label{Costabilization}
The \emphbf{costabilization} $\cR(\underline{\cA})$ of $\underline{\cA}$ is a category with objects consisting of sequences $(\underline{A_n},\alpha_n)_{n\in \bZ}$ where $\underline{A_n}\in \underline{\cA}$ and $\alpha_n\colon \underline{A_n}\xrightarrow{\cong} \Omega \underline{A_{n+1}}$ is an isomorphism in $\underline{\cA}$. A morphism 
\[
(\underline{A_n},\alpha_n)\to (\underline{B_n},\beta_n)
\]
in $\cR(\underline{\cA})$ consists of a sequence $(f_n)_{n\in \bZ}$ of morphisms $f_n\colon \underline{A_n}\to \underline{B_n}$ in $\underline{\cA}$ satisfying $\beta_n\circ f_n = \Omega(f_{n+1})\circ \alpha_n$.  
\end{Definition}

\begin{Remark}\label{universal property costabilization}
Here we explain the name and the universal property of the costabilization. We follow the conventions in \cite{Hel68}. A \emphbf{category with suspension} is a pair $(\cC,T)$ where $\cC$ is a category and $T\colon \cC\to \cC$ is a functor. A \emphbf{weakly stable} morphism 
\[
(F,\phi)\colon (\cC_1,T_1)\to (\cC_2,T_2)
\] 
between two categories with suspension is given by a functor $F\colon \cC_1\to \cC_2$ together with an isomorphism $\phi\colon F\circ T_1\xrightarrow{\cong} T_2\circ F$. It is called \emphbf{stable} if $\phi$ is the identity morphism. Composition of weakly stable morphisms is given by 
\[
(G,\psi)\circ (F,\phi)= (G\circ F, \psi F\circ G(\phi)).
\] 
This gives a category where the objects are categories with suspensions, and where the morphisms are the weakly stable morphisms. If $(\cC,T)$ is a category with suspension, then we say that $(\cC,T)$ is \emphbf{stable} if $T\colon \cC\to \cC$ is an autoequivalence. Given a category with suspension $(\cC,T)$, we can associate a stable category $(\cR(\cC,T),\hat{T})$, called its \emphbf{costabilization}, as follows: An object of $\cR(\cC,T)$ is a sequence $(C_n,\alpha_n)_{n\in \bZ}$ where $C_n\in \cC$ and $\alpha_n\colon C_n\xrightarrow{\cong} T C_{n+1}$ is an isomorphism in $\cC$, and a morphism $(C_n,\alpha_n)\to (C'_n,\beta_n)$ in $\cR(\cC,T)$ is a sequence $(f_n)_{n\in \bZ}$ of morphisms $f_n\colon C_n\to C'_n$ in $\cC$ satisfying $\beta_n\circ f_n = T(f_{n+1})\circ \alpha_n$. The autoequivalence $\hat{T}\colon \cR(\cC,T)\to \cR(\cC,T)$ is given by 
\[
\hat{T}(C_n,\alpha_n)=(C_{n-1},\alpha_{n-1}).
\]
Note that if we consider $(\underline{\cA},\Omega)$ as a category with suspension, then we have that $\cR(\underline{\cA},\Omega)=\cR(\underline{\cA})$ where $\cR(\underline{\cA})$ is as in Definition \ref{Costabilization}. Now for a category with suspension $(\cC,T)$ there exists a weakly stable morphism  
\[
(R,\gamma)\colon (\cR(\cC,T),\hat{T})\to (\cC,T)
\]
 where $R \colon \cR(\cC,T) \to \cC$ is the forgetful functor sending $(C_n,\alpha_n)$ to $C_0$, and $\gamma\colon R\circ \hat{T} \xrightarrow{\cong} T\circ R$ is the isomorphism given by 
\[
R\hat{T}(C_n,\alpha_n) = C_{-1}\xrightarrow{\alpha_{-1}} T(C_0) = T R(C_n,\alpha_n)
\] 
The costabilization satisfies the following universal lifting property: If 
\[(F,\mu)\colon (\cB,\Sigma)\to (\cC,T)
\]
is a weakly stable morphism and $(\cB,\Sigma)$ is stable, then there exists a unique stable morphism $(G,1)\colon (\cB,\Sigma)\to (\cR(\cC,T),\hat{T})$ satisfying 
\[
(R,\gamma) \circ (G,1)=(F,\mu).
\] 
Explicitly, $G$ is given by $G(B) = (B_n,\beta_n)$ where $B_n=F\Sigma^{-n}(B)$ and $\beta_n\colon B_n \xrightarrow{\cong} TB_{n+1}$ is given by 
\[
B_n=F\Sigma^{-n}(B)\xrightarrow{\mu\Sigma^{-n-1}(B)}T F\Sigma^{-n-1}(B) =TB_{n+1}.
\]
This lifting property is dual to the universal extension property for the stabilization \cite[Proposition 1.1]{Hel68}, whence the name costabilization.
\end{Remark}

We fix some notation. Let $\textbf{C}(\cA)$ be the category of complexes in $\cA$. An object in $\textbf{C}(\cA)$ is denoted by
\[
(P_{\bullet},d_{\bullet}):=\cdots \xrightarrow{d_{-2}}P_{-1}\xrightarrow{d_{-1}} P_0\xrightarrow{d_0} P_1\xrightarrow{d_1}\cdots
\] 
For each integer $n\in \bZ$  we have functors 
\[
Z_n(-)\colon \textbf{C}(\cA)\to \cA \quad \text{and} \quad H_n(-)\colon \textbf{C}(\cA)\to \cA
\]
given by taking the $n$th cycles $Z_n(P_{\bullet},d_{\bullet})=\Ker d_n$ and the $n$th homology $H_n(P_{\bullet},d_{\bullet}):= \Ker d_n/\im d_{n-1}$. We say that $(P_{\bullet},d_{\bullet})$ is acyclic if $H_n(P_{\bullet},d_{\bullet})=0$ for all $n\in \bZ$. We call a morphism $(P_{\bullet},d_{\bullet})\xrightarrow{f_\bullet}(Q_{\bullet},d'_{\bullet})$ of complexes null-homotopic if there exists morphisms $h_i\colon P_i\to Q_{i-1}$ in $\cA$ such that $f_i=d'_{i-1}\circ h_i + h_{i+1}\circ d_i$ for all $i\in \bZ$. Let $\textbf{C}_{ac}(\cP)$ denote the full subcategory of $\textbf{C}(\cA)$ consisting of acyclic complexes with projective components, and let $\textbf{K}_{ac}(\cP)$ denote the homotopy category of $\textbf{C}_{ac}(\cP)$.  Explicitly, $\textbf{K}_{ac}(\cP)$ has the same objects as $\textbf{C}_{ac}(\cP)$, and with morphism spaces
\[
\textbf{K}_{ac}(\cP)((P_{\bullet},d_{\bullet}),(Q_{\bullet},d'_{\bullet})) = \textbf{C}_{ac}(\cP)((P_{\bullet},d_{\bullet}),(Q_{\bullet},d'_{\bullet}))/\sim
\]
where $f_{\bullet}\sim g_{\bullet}$ if the difference $f_{\bullet}-g_{\bullet}$ is null-homotopic.

Given $(P_{\bullet},d_{\bullet})$ in $\textbf{C}_{ac}(\cP)$, we obtain an object $(\underline{Z_n(P_{\bullet},d_{\bullet})},\alpha_n )$ in $\cR(\underline{\cA})$ where 
\[
\alpha_n\colon \underline{Z_n(P_{\bullet},d_{\bullet})} \to \Omega\underline{Z_{n+1}(P_{\bullet},d_{\bullet})}
\]
is the induced isomorphism as in \eqref{unique syzygy}. Furthermore, given a morphism $f_{\bullet}\colon (P_{\bullet},d_{\bullet})\xrightarrow{}(Q_{\bullet},d'_{\bullet})$ in $\textbf{C}_{ac}(\cP)$ we obtain morphisms 
\[
\underline{Z_n(f_{\bullet})}\colon \underline{Z_n(P_{\bullet},d_{\bullet})}\xrightarrow{}\underline{Z_n(Q_{\bullet},d'_{\bullet})}
\]
in $\underline{\cA}$ for each $n\in \bZ$, and it is easy to see that they make the diagram 
\[
\begin{tikzpicture}[description/.style={fill=white,inner sep=2pt}]
\matrix(m) [matrix of math nodes,row sep=2.5em,column sep=4em,text height=1.5ex, text depth=0.25ex] 
{ \underline{Z_n(P_{\bullet},d_{\bullet})} & \Omega\underline{Z_{n+1}(P_{\bullet},d_{\bullet})} \\
  \underline{Z_n(Q_{\bullet},d'_{\bullet})} & \Omega\underline{Z_{n+1}(Q_{\bullet},d'_{\bullet})} \\};
\path[->]
(m-1-1) edge node[auto] {$\cong$} 	    										(m-1-2)

(m-2-1) edge node[auto] {$\cong$} 	    										(m-2-2)

(m-1-1) edge node[auto] {$\underline{Z_n(f_{\bullet})}$} 	    				(m-2-1)
(m-1-2) edge node[auto] {$\Omega \underline{Z_{n+1}(f_{\bullet})}$} 	    		(m-2-2);	
\end{tikzpicture}
\]
commute where the horizontal isomorphisms are as in \eqref{unique syzygy}. Hence, we obtain a morphism $(\underline{Z_n(f_{\bullet})}\colon \underline{Z_n(P_{\bullet},d_{\bullet})}\xrightarrow{}\underline{Z_n(Q_{\bullet},d'_{\bullet})})_{n\in \bZ}$ in $\cR(\underline{\cA})$, and therefore, we have a functor 
\[
\textbf{C}_{ac}(\cP)\to \cR(\underline{\cA}).
\]
If $f_{\bullet}$ is null-homotopic, then the morphism $Z_n(f_{\bullet})\colon Z_n(P_{\bullet},d_{\bullet})\xrightarrow{}Z_n(Q_{\bullet},d'_{\bullet})$ factors through $P_n$, and hence $\underline{Z_n(f_{\bullet})}=0$. Therefore, we get an induced functor
\[
\textbf{K}_{ac}(\cP)\to \cR(\underline{\cA})
\] 

\begin{Proposition}\label{Costabilization vs acyclic complex with projective components}
The functor $\textbf{K}_{ac}(\cP)\to \cR(\underline{\cA})$ is dense and full. \footnote{This functor is claimed to be an equivalence in Theorem 3.11 in \cite{Bel00}. It is not clear to the authors why this is true.}
\end{Proposition}

\begin{proof}
Let $(\underline{A_n},\alpha_n)$ be an arbitrary object in $\cR(\underline{\cA})$. By assumption, for all $n\in \bZ$ there exists an object $P_{n-1}\in \cP$ and an exact sequence
\[
0\to \Omega A_{n}\to P_{n-1}\to A_{n}\to 0 
\]
Since $\underline{A_n}\cong \Omega\underline{A_{n+1}}$ in $\underline{\cA}$, there exists objects $P_{n-1}',P_{n-1}''\in \cP$ and an isomorphism $A_n\oplus P_{n-1}'\cong \Omega A_{n+1}\oplus P_{n-1}''$ in $\cA$. This implies that there also exists an exact sequence
\[
0\to \Omega A_{n}\xrightarrow{} Q_{n-1}\xrightarrow{} \Omega A_{n+1}\to 0
\]
in $\cA$ where $Q_{n-1}\in \cP$. Hence, we obtain a complex $(Q_{\bullet},d_{\bullet})$ in $\textbf{K}_{ac}(\cP)$ with differential $d_n$ given by the composite $Q_n\to \Omega A_{n+2} \to Q_{n+1}$. Furthermore,  by construction the image of the complex $(Q_{\bullet},d_{\bullet})$ under the functor $\textbf{K}_{ac}(\cP)\to \cR(\underline{\cA})$ is the object $(\Omega \underline{A_{n+1}},\Omega(\alpha_{n+1}))$. Since we have an isomorphism
\[
(\Omega \underline{A_{n+1}},\Omega(\alpha_{n+1}))\cong (\underline{A_n},\alpha_n)
\]
in $\cR(\underline{\cA})$ it follows that the functor $\textbf{K}_{ac}(\cP)\to \cR(\underline{\cA})$  is dense.  

Let $(Q_{\bullet},d_{\bullet})$ and $(Q'_{\bullet},d'_{\bullet})$ be complexes in $\textbf{K}_{ac}(\cP)$ and let $A_n=Z_n(Q_{\bullet},d_{\bullet})$ and $A'_n=Z_n(Q'_{\bullet},d'_{\bullet})$ so that we have short exact sequences
\begin{align*}
& 0\to A_{n}\xrightarrow{i_n} Q_n\xrightarrow{p_n} A_{n+1}\to 0 \\
& 0\to A'_{n}\xrightarrow{i'_n} Q'_n\xrightarrow{p'_n} A'_{n+1}\to 0
\end{align*}
where $d_n= i_{n+1}\circ p_n$ and $d'_n= i'_{n+1}\circ p'_n$. Under the functor $\textbf{K}_{ac}(\cP)\to \cR(\underline{\cA})$  these complexes correspond to objects $(\underline{A_n},\alpha_n)$ and $(\underline{A'_n},\alpha'_n)$ in $\cR(\underline{\cA})$. Let 
\[
(\underline{f_n})\colon (\underline{A_n},\alpha_n)\to (\underline{A'_n},\alpha'_n)
\]
 be an arbitrary morphism between these objects in $\cR(\underline{\cA})$. For each $n\in \bZ$ choose a lifting $g_{n}\colon Q_{n}\to Q'_{n}$ of $f_{n+1}\colon A_{n+1}\to A'_{n+1}$. Since $A_{n}=\Ker p_n$ and $A'_n= \Ker p'_n$, we get a unique morphism $k_{n}\colon  A_{n}\to  A'_{n}$ satisfying $i'_{n}\circ k_n=g_{n}\circ i_{n}$. It is easy to see that $\underline{k_{n}}=\underline{f_n}$, and hence there exists a morphism $h_n\colon A_{n}\to Q'_{n-1}$ such that 
 \[
p'_{n-1}\circ h_n=f_{n}-k_{n}.
\]
Now since 
\begin{align*}
d'_n\circ (g_n-h_{n+1}\circ p_n) & =d'_n\circ g_n - d'_n\circ h_{n+1}\circ p_n 
 \\
& =i'_{n+1}\circ f_{n+1} \circ p_n-i'_{n+1}\circ (f_{n+1}-k_{n+1})\circ p_n \\
& = i'_{n+1}\circ k_{n+1}\circ p_n 
\end{align*}
and
\begin{align*}
(g_{n+1}-h_{n+2}\circ p_{n+1})\circ d_n = g_{n+1}\circ d_n = g_{n+1}\circ i_{n+1}\circ p_n \\ = i'_{n+1}\circ k_{n+1}\circ p_n
\end{align*}
it follows that the maps $l_n=g_n-h_{n+1}\circ p_n\colon Q_n\to Q'_n$ for all $n\in \bZ$ induce a map of chain complexes $l_{\bullet}\colon (Q_{\bullet},d_{\bullet})\to (Q'_{\bullet},d'_{\bullet})$. Since $Z_n(l_{\bullet})=k_n$ and $\underline{k_n}= \underline{f_n}$, it follows that the functor $\textbf{K}_{ac}(\cP)\to \cR(\underline{\cA})$ is full. 
\end{proof}

\begin{Remark}
It would be interesting to determine if the functor $\textbf{K}_{ac}(\cP)\to \cR(\underline{\cA})$ is an equivalence in general, or to find a counterexample and to determine  in which cases it induces an equivalence.
\end{Remark}

Let $R\colon \cR(\underline{\cA})\to \underline{\cA}$ be the forgetful functor sending $(\underline{A_n},\alpha_n)$ to $\underline{A_0}$, and let 
\[
\im R =\{\underline{A}\in \underline{\cA}\mid \underline{A}\cong R(X) \text{ for some }X\in \cR(\underline{\cA})\}.
\]
denote the essential image of $R$.

\begin{Definition}\label{Definition Co-Gorenstein}
Let $\cA$ be an abelian category with enough projectives. We say that $\cA$ is $\cP$-\emphbf{Co-Gorenstein} if the following holds:
\begin{enumerate}
\item\label{Definition Co-Gorenstein:1} The forgetful functor $R\colon \cR(\underline{\cA})\to \underline{\cA}$ is full and faithful;
\item\label{Definition Co-Gorenstein:2} If $0\to A_1\to A_2\to A_3$ is an exact sequence in $\cA$ with $\underline{A_1},\underline{A_3}\in \im R$, then $\underline{A_2}\in \im R$.\footnote{In \cite[Definition 3.13]{Bel00} it is only required that $R$ is full and faithful, and it is claimed that this implies assumption \ref{Definition Co-Gorenstein:2}, see \cite[Proposition 2.13 part (1)]{Bel00}. This is not clear to the authors, so we include this assumption in the definition.}
\end{enumerate} 
\end{Definition}

The notion of Co-Gorenstein category was defined more generally for left triangulated categories in \cite[Definition 3.13]{Bel00} and for an exact category in \cite[Definition 4.9]{Bel00}. However, we only consider the case above.

\begin{Remark}
We explain the name Co-Gorenstein: Let $\cS(\underline{\cA},\Omega)$ be the stabilization of the pair $(\underline{\cA},\Omega)$, see \cite{Hel68}. By \cite[Proposition 1.1]{Hel68} there exists a functor $\underline{\cA}\to \cS(\underline{\cA},\Omega)$ which satisfies a universal extension property dual to the universal lifting property stated in Remark \ref{universal property costabilization} for $\cR(\underline{\cA})$. Following \cite{Bel00}, the category $\underline{\cA}$ is called $\cP$-Gorenstein if there exists a full left triangulated subcategory $\cV\subset \underline{\cA}$ such that the composite $\cV \to \underline{\cA} \to \cS(\underline{\cA},\Omega)$ is an equivalence of left triangulated categories, see \cite[Definition 3.13]{Bel00}. This coincides well with the terminology in the literature, since if $\cA = \Md\text{-}\Lambda$ where $\Lambda$ is a noetherian ring, then $\cA$ is $\cP$-Gorenstein if and only if $\Lambda$ is an Iwanaga-Gorenstein ring, i.e. if the left and right injective dimension of $\Lambda$ as a module over itself is finite \cite[Theorem 6.9 and Corollary 6.11]{Bel00}. Since the definition of $\cP$-Co-Gorenstein is in terms of the costabilization rather than then the stabilization, this can explain the name.    
\end{Remark}

Our goal in the remainder of this subsection is to give a different characterization of $\cP$-Co-Gorenstein categories. To this end, let 
\begin{multline*}
\Omega^{\infty}(\cA):= \{A\in \cA\mid\text{there exists an exact sequence} \\
0\to A\to P_0\to P_1\to \cdots \text{with } P_i\in \cP \text{ } \forall i\geq 0\}.
\end{multline*} 

\begin{Lemma}\label{Infinite syzygy}
Let $\cA$ be an abelian category with enough projectives. Then $X\in \im R$ if and only if there exists $A\in \Omega^{\infty}(\cA)$ such that $\underline{A}\cong X$.
\end{Lemma}

\begin{proof}
This follows immediately from Proposition \ref{Costabilization vs acyclic complex with projective components}.
\end{proof}

 The proof of the following result is essentially the same as in \cite[Theorem 3.3]{ABM98}, but for the convenience of the reader we reproduce the argument here.

\begin{Lemma}\label{syzygy isomorphism}
Let $\cA$ be an abelian category with enough projectives, and let $A\in \cA$. The following statements are equivalent:
\begin{enumerate}
\item\label{syzygy isomorphism:1} $\Ext^1_{\cA}(A,P)=0$ for all $P\in \cA$ projective;
\item\label{syzygy isomorphism:2} The natural map $\underline{\cA}(\underline{A},\underline{A'}) \to \underline{\cA}(\Omega \underline{A}, \Omega \underline{A'})$ is an isomorphism for all $A'\in \cA$;
\item\label{syzygy isomorphism:3} The natural map $\underline{\cA}(\underline{A},\underline{A'}) \to \underline{\cA}(\Omega \underline{A}, \Omega \underline{A'})$ is an isomorphism for all $A'\in \cA$ for which there exists an exact sequence 
\[
0\to P\to A'\to A\to 0
\] 
with $P\in \cA$ projective.
\end{enumerate}
\end{Lemma}

\begin{proof}
We prove that \ref{syzygy isomorphism:1} implies \ref{syzygy isomorphism:2}. Consider the exact sequences \\ $0\to \Omega A \xrightarrow{i} P\xrightarrow{p} A\to 0$ and $0\to \Omega A' \xrightarrow{i'} P'\xrightarrow{p'} A'\to 0$ with $P$ and $P'$ projective. Given a morphism $\Omega A\xrightarrow{f} \Omega A'$, the composite $\Omega A\xrightarrow{f} \Omega A'\xrightarrow{i'} P'$ can be extended to a morphism $h\colon P\to P'$ since $\Ext^1_{\cA}(A,P')=0$. This gives a commutative diagram   
\[
\begin{tikzpicture}[description/.style={fill=white,inner sep=2pt}]
\matrix(m) [matrix of math nodes,row sep=2.5em,column sep=4em,text height=1.5ex, text depth=0.25ex] 
{0 & \Omega A & P & A & 0\\
 0 &  \Omega A' & P' & A' & 0 \\};
\path[->]
(m-1-1) edge node[auto] {$$} 	    										(m-1-2)
(m-1-2) edge node[auto] {$i$} 	    										(m-1-3)
(m-1-3) edge node[auto] {$p$} 	    										(m-1-4)
(m-1-4) edge node[auto] {$$} 	    										(m-1-5)

(m-2-1) edge node[auto] {$$} 	    										(m-2-2)
(m-2-2) edge node[auto] {$i'$} 	    										(m-2-3)
(m-2-3) edge node[auto] {$p'$} 	    										(m-2-4)
(m-2-4) edge node[auto] {$$} 	    										(m-2-5)

(m-1-2) edge node[auto] {$f$} 	    						    			(m-2-2)
(m-1-3) edge node[auto] {$h$} 	    									  	(m-2-3)
(m-1-4) edge node[auto] {$g$} 	    								 		(m-2-4);	
\end{tikzpicture}
\]
with exact rows, where $g$ is induced from the commutativity of the left square. In particular, we get that $\Omega \underline{g} = \underline{f}$. This shows that the natural map $\underline{\cA}(\underline{A},\underline{A'}) \to \underline{\cA}(\Omega \underline{A}, \Omega \underline{A'})$ is an epimorphism. To prove that is a monomorphism, we assume that we are given a morphism $g\colon A\to A'$ as above, and that $\Omega \underline{g}=\underline{f}=0$. This means that $f$ can be written as a composite $\Omega A\xrightarrow{u} Q\xrightarrow{v} \Omega A'$ where $Q$ is projective. As $\Ext^1_{\cA}(A,Q)=0$ we get that $u$ comes from some $w\colon P\to Q$. Since $(h-i'vw)i=0$, there is a $t\colon A\to P$ such that $t\circ p = h-i'vw$. The morphism $t$ is then a lifting of $g$ to $P'$. In particular, $\underline{g}=0$. Hence, the map $\underline{\cA}(\underline{A},\underline{A'}) \to \underline{\cA}(\Omega \underline{A}, \Omega \underline{A'})$ is also a monomorphism, and therefore an isomorphism. 

Obviously, \ref{syzygy isomorphism:2} implies \ref{syzygy isomorphism:3}, so it only remains to show that \ref{syzygy isomorphism:3} implies \ref{syzygy isomorphism:1}. Let $0\to P \to A'\xrightarrow{f} A\to 0$ be an arbitrary exact sequence with $P$ projective. We want to show that it is split exact. Let $0\to \Omega A\to Q\to A\to 0$ be an exact sequence with $Q$ projective. The pullback of $f$ along $Q\to A$ shows that $\Omega \underline{f}\colon \Omega \underline{A'}\to \Omega \underline{A}$ is an isomorphism in the stable category. By assumption, if $\underline{h}$ is its inverse, then $\underline{h}=\Omega \underline{g}$ for some $g\colon A\to A'$. The same assumption applied to the pair $A,A$ shows that the map $\underline{\cA}(\underline{A},\underline{A}) \to \underline{\cA}(\Omega \underline{A}, \Omega \underline{A})$ is an isomorphism and therefore $\underline{f}\circ \underline{g}=\underline{1}$. But an epimorphism in a module category is split if and only if it is a split epimorphism in the stable category, see \cite[Proposition 5.4]{MZ15}. Thus $f$ is a split epimorphism. 
\end{proof}

Let 
\begin{align*}
& {}^{{}_1\perp}\cP := \{A\in \cA\mid\Ext^1_{\cA}(A,P)=0\text{ for all } P\in \cP\} \\
& {}^{\perp}\cP := \{A\in \cA\mid\Ext^i_{\cA}(A,P)=0\text{ for all } P\in \cP \text{ and } i\geq 1\}
\end{align*}
A complex $(P_{\bullet},d_{\bullet})$ in $\textbf{C}_{ac}(\cP)$ is called totally acyclic if the complex
\[
\cdots \xrightarrow{-\circ d_1}\cA(P_{1},Q)\xrightarrow{-\circ d_0} \cA(P_{0},Q)\xrightarrow{-\circ d_{-1}}\cdots  
\] 
is acyclic for any $Q\in \cP$. An object $A\in \cA$ is called Gorenstein projective if $A=Z_0(P_{\bullet},d_{\bullet})$ for some totally acyclic complex $(P_{\bullet},d_{\bullet})$. The subcategory of Gorenstein projective objects in $\cA$ is denoted by  $\Gproj(\cA)$. 

\begin{Corollary}[Theorem 4.10 in \cite{Bel00}]\label{Alternative description Co-Gorenstein}
Let $\cA$ be an abelian category with enough projectives. The following statements are equivalent:
\begin{enumerate}
\item\label{Alternative description Co-Gorenstein:1} $\cA$ is $\cP$-Co-Gorenstein;
\item\label{Alternative description Co-Gorenstein:2} $\Omega^{\infty}(\cA)\subset {}^{{}_1\perp}\cP$;
\item\label{Alternative description Co-Gorenstein:3} $\Omega^{\infty}(\cA)\subset {}^{\perp}\cP$; 
\item\label{Alternative description Co-Gorenstein:4} $\Omega^{\infty}(\cA)=\Gproj(\cA)$.
\end{enumerate}
\end{Corollary}

\begin{proof}
Let $\hat{\Omega}\colon \cR(\underline{\cA})\to \cR(\underline{\cA})$ be the autoequivalence given by $\hat{\Omega}(\underline{A_n},\alpha_n)=(\underline{A_{n-1}},\alpha_{n-1})$. Then there exists an isomorphism $R\circ \hat{\Omega} \cong \Omega \circ R$. Hence, if $\cA$ is $\cP$-Co-Gorenstein, then $R$ is an equivalence onto $\im R$, and therefore $\Omega\colon \im R\to \im R$ is also an equivalence. It follows that $\Omega^{\infty}(\cA)\subset {}^{{}_1\perp}\cP$ by Lemma \ref{Infinite syzygy} and Lemma \ref{syzygy isomorphism}, which proves \ref{Alternative description Co-Gorenstein:1} $\implies$ \ref{Alternative description Co-Gorenstein:2}. The implications \ref{Alternative description Co-Gorenstein:2} $\implies$ \ref{Alternative description Co-Gorenstein:3} and \ref{Alternative description Co-Gorenstein:3} $\implies$ \ref{Alternative description Co-Gorenstein:4} are straightforward. For \ref{Alternative description Co-Gorenstein:4} $\implies$ \ref{Alternative description Co-Gorenstein:1}, note first that $\Omega \colon \im R\to \im R$ is full and faithful by Lemma \ref{Infinite syzygy} and Lemma \ref{syzygy isomorphism}. Let $(f_n)\colon (\underline{A_n},\alpha_n)\to (\underline{A'_n},\alpha'_n)$ be a morphism in $\cR(\underline{\cA})$. For $n<0$ we can write $f_n$ as a composite 
\begin{align}\label{equation 1}
 \underline{A_n}\cong \Omega(\underline{A_{n+1}})\cong \cdots \cong \Omega^{-n}(\underline{A_0})\xrightarrow{\Omega^{-n}(f_0)} \Omega^{-n}(\underline{A'_0})\cong \cdots \cong \underline{A'_n}
\end{align}
and for $n>0$ we can write $\Omega^n(f_n)$ as a composite
\begin{align}\label{equation 2}
\Omega^n(\underline{A_n})\cong \Omega^{n-1}(\underline{A_{n-1}})\cong \cdots \cong \underline{A_0}\xrightarrow{f_0} \underline{A'_0}\cong \cdots \cong \Omega^n(\underline{A'_n})
\end{align}
Hence, if $f_0=0$ then $f_n=0$ for $n<0$ and $\Omega^n(f_n)=0$ for $n>0$. Since $\Omega$ is faithful, it follows that $f_n=0$ for all $n\in \bZ$, and therefore $R$ is faithful. To show that $R$ is full, we chose again two objects $(\underline{A_n},\alpha_n)$ and $(\underline{A'_n},\alpha'_n)$ in $\cR(\underline{\cA})$, and we let $f_0\colon \underline{A_0}\to \underline{A'_0}$ be an arbitrary morphism in $\underline{\cA}$. Define morphisms $f_n\colon \underline{A_{n}}\to \underline{A'_n}$ for $n<0$ and $g_n\colon \Omega^n(\underline{A_n})\to \Omega^n(\underline{A'_n})$ for $n>0$ in $\underline{\cA}$ by equation \eqref{equation 1} and \eqref{equation 2}, respectively. Since $\Omega$ is full and faithful, there exists for each $n>0$ a unique morphism $f_n\colon \underline{A_{n}}\to \underline{A'_n}$ satisfying $\Omega^n(f_n)=g_n$. A straightforward computation then shows that $(f_n)\colon (\underline{A_n},\alpha_n)\to (\underline{A'_n},\alpha'_n)$ is a morphism in $\cR(\underline{\cA})$, and hence $R$ is full. Finally, part \ref{Definition Co-Gorenstein:2} in the definition of $\cP$-Co-Gorenstein holds since $\Gproj(\cA)$ is closed under extensions and by Lemma \ref{Infinite syzygy}. Hence, the claim follows. 
\end{proof}

\section{Co-Gorenstein Artin algebras}\label{Co-Gorenstein Artin algebras}

We now restrict ourselves to the case where $\cA=\md \text{-}\Lambda$ and $\cP=\Proj (\md \text{-}\Lambda)$ for an artin $R$-algebra $\Lambda$. 

\begin{Definition}\label{Definition right Co-Gorenstein}
$\Lambda$ is \emphbf{right Co-Gorenstein} if $\md\text{-}\Lambda$ is $\cP$-Co-Gorenstein.
\end{Definition}

By the above results we know that $\Lambda$ is right Co-Gorenstein if and only if one of the following equivalent conditions hold:

\begin{enumerate}
\item $\Omega^{\infty}(\md\text{-}\Lambda)\subset {}^{{}_1\perp}\Lambda$;
\item $\Omega^{\infty}(\md\text{-}\Lambda)\subset {}^{\perp}\Lambda$;
\item $\Omega^{\infty}(\md\text{-}\Lambda)=\Gproj(\md\text{-}\Lambda)$.
\end{enumerate}
 
Note that any Iwanaga-Gorenstein algebra is Co-Gorenstein. The following example shows that the converse is not true.

\begin{Example}
Let $\Lambda:=k[x,y]/(x^2,xy,yx,y^2)$, and let $S$ be the unique simple $\Lambda$-module. $\Lambda$ is a 3-dimensional local algebra with a two dimensional socle, and therefore $\Lambda$ is not an Iwanaga-Gorenstein algebra as a local artin algebra is a Iwanaga-Gorenstein algebra if and only if it has simple socle. Note that 
\[
\Omega^1(\md\text{-}\Lambda)=\add S\oplus \Lambda,
\]
because $\Lambda$ is a radical square zero algebra and thus every kernel of a projective cover is semisimple.
Hence, if $M\in \Omega^{\infty}(\md\text{-}\Lambda)$ then $M\cong \Lambda^{n}\oplus S^m$ for $m,n\geq 0$. Note that in a local algebra, every module has projective dimension zero or infinite and thus the cokernel of a monomorphism of the form $\Lambda^n\to \Lambda^r$ is projective. Therefore, any monomorphism $\Lambda^n\to \Lambda^r$ is split. It follows that $S^m \in \Omega^{\infty}(\md\text{-}\Lambda)$. On the other hand, if there exists an exact sequence
\[
0\to S^{m_1}\to \Lambda^{m_2}\to S^{m_3}\to 0
\]
then we must have $m_1=2m_2=2m_3$. In particular, we have that $S^m\notin \Omega^s(\md\text{-}\Lambda)$ if $0<m<2^{s-1}$. Since $S^m \in \Omega^{\infty}(\md\text{-}\Lambda)$ we must have that $m=0$ and hence
\[
\Omega^{\infty}(\md\text{-}\Lambda)=\add \Lambda\subseteq {}^{\perp}\Lambda.
\]
Therefore, $\Lambda$ is right Co-Gorenstein. Finally, note that 
\[
\bigoplus_{i\in \bZ}S\in \Omega^{\infty}(\Md\text{-}\Lambda).
\]
where $\Md\text{-}\Lambda$ is the category of all right $\Lambda$-modules, not necessarily finite dimensional. Since Gorenstein projective modules are closed under direct summands and $S$ is not Gorenstein projective, it follows that $\Md\text{-}\Lambda$ is not $\cP$-Co-Gorenstein. 
\end{Example}

Our goal now is to prove the implications between the conjectures. Fix a minimal projective resolution 
\[
\cdots \to P_1(D\Lambda)\to P_0(D\Lambda)\to D\Lambda\to 0
\] 
and a minimal injective resolution  
\[
0\to \Lambda\to I_0(\Lambda)\to I_1(\Lambda)\to \cdots
\] 
of $\Lambda$ as a right module. Let $\cX_n:=\add \Omega^n(\md\text{-}\Lambda)$ denote the smallest additive subcategory of $\md\text{-}\Lambda$ which contains $\Omega^n(\md\text{-}\Lambda)$ and is closed under direct summands. Note that $\cX_n \neq \Omega^n(\md\text{-}\Lambda)$ in general, see the example after Proposition 3.5 in \cite{AR94}.

\begin{Theorem}\label{When d-syzygies are extension closed}
The following are equivalent for $n\geq 1$:
\begin{enumerate}
\item $\Omega^k(\md\text{-}\Lambda)$ is extension-closed for $1\leq k\leq n$;
\item $\cX_k$ is extension-closed for $1\leq k\leq n$;
\item $\idim P_k(D\Lambda)\leq k+1$ for $0\leq k<n$.
\end{enumerate}
\end{Theorem}

\begin{proof}
This follows from \cite[Theorem 4.7]{AR96}.
\end{proof}

Unfortunately, the conditions in Theorem \ref{When d-syzygies are extension closed} are not left-right symmetric, see the paragraph after Corollary 2.8 in \cite{AR94}. However, the following result shows that after a small modification one obtains a symmetric condition.

\begin{Theorem}\label{Left-right symmetric}
Let $n\geq 1$ be an integer. The following are equivalent:
\begin{enumerate}
\item $\idim P_k(D\Lambda)\leq k$ for all $0\leq k<n$;
\item $\pdim I_k(\Lambda)\leq k$ for all $0\leq k<n$.
\end{enumerate}
\end{Theorem}

\begin{proof}
This follows from \cite[Theorem 3.7]{FGR75}.
\end{proof}

We now show that \emphbf{Conjecture} \ref{Conjecture} implies \emphbf{NC}

\begin{Proposition}\label{lemma implies Nakayama lemma}
The following holds:
\begin{enumerate}
\item\label{Conjecture implies Nakayama lemma part (i)} If $\domdim \Lambda = \infty$ and \emphbf{Conjecture} \ref{Conjecture} holds, then $\Lambda$ is right Co-Go\-renstein;
\item\label{Conjecture implies Nakayama lemma part (ii)} If $\domdim \Lambda = \infty$ and $\Lambda$ is right Co-Gorenstein, then $\Lambda$ is selfinjective; 
\item\label{Conjecture implies Nakayama lemma part (iii)} \emphbf{Conjecture} \ref{Conjecture} implies \emphbf{NC}.
\end{enumerate}
\end{Proposition}

\begin{proof}
Part \ref{Conjecture implies Nakayama lemma part (i)} follows from Theorem \ref{When d-syzygies are extension closed} and Theorem \ref{Left-right symmetric}. We prove part \ref{Conjecture implies Nakayama lemma part (ii)}. Let $i\colon \Lambda \to I_0(\Lambda)$ denote the injective envelope. We have exact sequences 
\begin{equation}\label{sequence 1}
0\to \Lambda \xrightarrow{i} I_0(\Lambda)\to \Coker i\to 0
\end{equation}
and
\begin{equation}\label{sequence 2}
0\to \Coker i \to I_1(\Lambda)\to I_2(\Lambda)\to \cdots.
\end{equation}
Note that $\Coker i \in \Omega^{\infty}(\md\text{-}\Lambda)$ if $\domdim \Lambda = \infty$. Furthermore, $\Lambda$ is selfinjective if and only if the sequence \eqref{sequence 1} is split, and this holds if $\Coker i\in {}^{\perp}\Lambda$. By Corollary \ref{Alternative description Co-Gorenstein}, this proves part \ref{Conjecture implies Nakayama lemma part (ii)}. Part \ref{Conjecture implies Nakayama lemma part (iii)} follows from part \ref{Conjecture implies Nakayama lemma part (i)} and \ref{Conjecture implies Nakayama lemma part (ii)}.
\end{proof}

We now show that \emphbf{GNC} implies \emphbf{Conjecture} \ref{Conjecture}.

\begin{Proposition}\label{finite injective dimension}
The following holds:
\begin{enumerate}
\item\label{finite injective dimension (i)} Suppose the \emphbf{GNC} holds. If $\Omega^n(\md\text{-}\Lambda)$ is extension closed for all $n\geq 1$, then $\idim \Lambda < \infty$ as a right $\Lambda$-module;
\item\label{finite injective dimension (ii)} If $\idim \Lambda < \infty$ as right $\Lambda$-module, then $\Lambda$ is right Co-\-Gorenstein;
\item\label{finite injective dimension (iii)} \emphbf{GNC} implies \emphbf{Conjecture} \ref{Conjecture}.
\end{enumerate}

\end{Proposition}

\begin{proof}
By Theorem \ref{When d-syzygies are extension closed} we have that $\idim P_n(D\Lambda)\leq n+1$ for all $n\geq 0$. Now write $\Lambda = P_0 \oplus \cdots \oplus P_m$ as a sum of indecomposable projective $\Lambda$-modules. Since \emphbf{GNC} holds, there exists integers $s_0,s_1,\cdots,s_m$ such that $P_i$ is a direct summands of $P_{s_i}(D\Lambda)$. Let $s:=\max\{s_0,\cdots ,s_m\}+1$. Then $\idim P_i\leq \idim P_{s_i}(D\Lambda)\leq s_i+1 \leq s$ for all $0\leq i\leq m$. Hence, it follows that $\idim \Lambda \leq s < \infty$. 

For part \ref{finite injective dimension (ii)}, assume $\idim \Lambda=s$, and let $M\in \Omega^{\infty}(\md\text{-}\Lambda)$. Then there exists an exact sequence
\[
0\to M\to P_1\to \cdots \to P_s\to K\to 0
\]  
in $\md\text{-}\Lambda$ with $P_i$ projective. It follows by dimension shifting that 
\[
\Ext^i_{\Lambda}(M,\Lambda)\cong \Ext^{i+s}_{\Lambda}(K,\Lambda).
\]
Since $\idim \Lambda=s$, we have $\Ext^{i+s}_{\Lambda}(K,\Lambda)=0$ for all $i\geq 1$. This shows that $M\in {}^{\perp}\Lambda$, and hence $\Lambda$ is Co-Gorenstein by Corollary \ref{Alternative description Co-Gorenstein}. Part \ref{finite injective dimension (iii)} follows immediately from part \ref{finite injective dimension (i)} and \ref{finite injective dimension (ii)}.
\end{proof}

\begin{Remark}
In Propositon \ref{finite injective dimension} part \ref{finite injective dimension (ii)} we actually prove that 
\[
\cap_{n\geq 1} \Omega^{n}(\md\text{-}\Lambda)\subseteq {}^{\perp}\Lambda.
\]
However, under the assumption that $\Omega^n(\md\text{-}\Lambda)$ is extension closed for all $n\geq 1$, we have that 
\[
\Omega^{\infty}(\md\text{-}\Lambda) = \cap_{n\geq 1} \Omega^{n}(\md\text{-}\Lambda).\footnote{Beligiannis claims this holds without any extra assumptions on $\Lambda$, see the paragraph before Theorem 3.17 in \cite{Bel00}. The authors do not see why this is true.}
\]
In fact, by \cite[Theorem 1.7 part b) and c)]{AR96} we get that the modules in $\cap_{n\geq 1} \Omega^{n}(\md\text{-}\Lambda)$ can be identified with the modules which are $n$-torsion free for all $n$, and it is easy to see that these modules are contained in $\Omega^{\infty}(\md\text{-}\Lambda)$. 
\end{Remark}

\bibliography{Mybibtex}
\bibliographystyle{plain} 

\end{document}